\documentclass[12pt,letterpaper]{amsart}
\usepackage{amsmath,amssymb,amsfonts}
\usepackage[all]{xy}
\usepackage{graphicx}
\usepackage{caption}
\usepackage{enumerate}
\usepackage{xcolor}

\usepackage{appendix}
\usepackage{a4wide}
\usepackage{hyperref}

\newtheorem{thm}{Theorem} 
\newtheorem{prop}{Proposition}
\newtheorem{lem}{Lemma}
\newtheorem{cor}{Corollary}

\theoremstyle{definition}

\newtheorem{expl}{Example}

\newtheorem{question}{Question}
\newtheorem{rem}{Remark}

\DeclareMathOperator{\rank}{\text{rank}}

\DeclareMathOperator{\alg}{\text{alg}}

    \DeclareFontFamily{U}{wncy}{}
    \DeclareFontShape{U}{wncy}{m}{n}{<->wncyr10}{}
    \DeclareSymbolFont{mcy}{U}{wncy}{m}{n}
    \DeclareMathSymbol{\Sha}{\mathord}{mcy}{"58}

\numberwithin{equation}{section}

\DeclareSymbolFont{bbold}{U}{bbold}{m}{n}
\DeclareSymbolFontAlphabet{\mathbbold}{bbold}

\DeclareSymbolFont{bbold}{U}{bbold}{m}{n}
\DeclareSymbolFontAlphabet{\mathbbold}{bbold}

\title{Spectral perturbation by rank one matrices}
\author{Jon Merzel, J\'an Min\'a\v{c}, Lyle Muller, Federico W. Pasini, Tung T. Nguyen}
\address{Department of Mathematics, Soka University of America, 1 University Drive, Aliso Viejo, CA 92656}
\email{jmerzel@soka.edu} 

\address{Department of Mathematics, Western University, London, Ontario, Canada N6A 5B7}
\email{minac@uwo.ca}
\address{Brain and Mind Institute and Department of Mathematics, The University of Western Ontario, London, ON, Canada, N6A 5B7}
\email{f.pasini1@campus.unimib.it}

\address{Brain and Mind Institute and Department of Mathematics, The University of Western Ontario, London, ON, Canada, N6A 5B7}
\email{lmuller2@uwo.ca}

\address{Brain and Mind Institute and Department of Mathematics, The University of Western Ontario, London, ON, Canada, N6A 5B7}
\email{tungnt@uchicago.edu}

\date{\today}

\thanks{JM is partially supported  by the Natural Sciences and Engineering Research Council of Canada (NSERC) grant R0370A01. JM also gratefully acknowledges Faculty of Science Distinguished Research Professorship for 2020/21. 
}
 
\date{\today}

\begin{document}
\maketitle
\begin{abstract}
Let $A$ be a matrix of size $n \times n$ over an algebraically closed field $F$ and $q(t)$ a monic polynomial of degree $n$. In this article, we describe the necessary and sufficient conditions of $q(t)$ so that there exists a rank one matrix $B$ such that the characteristic polynomial of $A+B$ is $q(t)$.
    
\end{abstract}

\section{Introduction and main results} 

Let $A$ be an $n \times n$ matrix over an algebraically closed field $F$. The eigenspectrum of matrices of the form $A+B$ where $B$ is a low rank matrix has been studied extensively in the literature (for example, see \cite{[Bau]}, \cite{[Kato]}, \cite{[Lidskii]}, \cite{[MMRR]}). The case $B$ has rank at most one is particularly interesting. In \cite[Theorem 1.1]{[CNg]}, the authors prove the following.

\begin{thm} \label{thm: CNg} (see \cite[Theorem 1.1]{[CNg]}) Let $A$ be a matrix with characteristic polynomial $p_{A}(t)=\prod_{i=1}^n (t-\lambda_i)$. Let $q(t)$ be a monic polynomial of degree $n$ and $(a_1, \ldots, a_n) \in F^n$ such that

\[ \frac{q(t)}{p_{A}(t)}=1+\sum_{i=1}^n \frac{a_i}{t-\lambda_i} .\] 
Then there exists a matrix $B$ of rank at most one such that the characteristic polynomial of $A+B$ is $q(t)$.

\end{thm}

This is a quite interesting theorem. This leads to the following question.

\begin{question}
Suppose $A$ is given. Find the necessary and sufficient conditions on $q(t)$ so that there exists a rank $1$ matrix $B$ such that the characteristic polynomial of $A+B$ is $q(t)$.
\end{question} 

In this article, we solve this question completely. To state the main theorem, we introduce some notation. Let $\lambda \in F$. For a polynomial $f(t) \in F[t]$, we define $m_{\lambda}(f)$ to be the multiplicity of $(t-\lambda)$ in $f(t)$. For a matrix $A$, we denote by $p_A(t)$ its characteristic polynomial, and by $\alg_{\lambda}(A)$ the algebraic multiplicity of $\lambda$ as an eigenvalue of $A$, namely 
\[ \alg_{\lambda}(A)= m_{\lambda}(p_{A}(t)).\]
Finally, we let $j_{\lambda}(A)$ to be the size of the largest Jordan block in $A$ with $\lambda$ on the main diagonal. Our main theorem is the following.

\begin{thm} \label{thm:main}
Let $A$ be an $n \times n$ matrix and $q(t)$ is a monic polynomial of degree $n$. Then there exists a rank one matrix $B$ such that 
\[ p_{A+B}(t)=q(t) ,\] 
if and only if for all $\lambda \in F$, the following condition is satisfied 
\[ m_{\lambda}(q) \geq \alg_{\lambda}(A)-j_{\lambda}(A) .\]
\end{thm}

\subsection{Acknowledgment}
The last author (Tung T. Nguyen) would like to thank Prof. Christian Mehl and Prof. Fuzhen Zhang for their interesting correspondences. He is also grateful to Prof. Christopher Godsil for teaching him Lemma \ref{lem:WA_formula} and Lemma \ref{lem:factorization} which play a key role in this article. The topic of this article arose when we studied some problems in spectral graph theory during the 2021 Fields Undergraduate Summer Research Program. We want to thank the participants for several inspiring discussions and the Fields Institute for providing an excellent working environment.

\section{Necessary conditions} 
In this section, we describe the necessary conditions for Theorem \ref{thm:main}. We first start with a simple observation.  By theorem \ref{thm: CNg}, if $q(t)$ satisfies 
\[ m_{\lambda}(q)+1 \geq m_{\lambda}(p) , \forall \lambda \in F \]
then there will be a rank 1 matrix $B$ such that the characteristic polynomial of $B+A$ is $q(t)$. While the above condition is sufficient, there are examples that do not have this property. Here is one particular example.
\begin{expl} \label{expl:first_expl}
Let 
\[ A=\begin{pmatrix} 0 & 1 \\ 0 & 0\end{pmatrix} .\]
Let 
\[ B=\begin{pmatrix} a_{11} & a_{12} \\ a_{21} & a_{22} \end{pmatrix} .\] 
Note that $B$ has rank at most $1$ if and only if $a_{11}a_{22}-a_{12}a_{21}=\det(B)=0$. We then have 
\[ p_{A+B}(t)=t^2-(a_{11}+a_{22})t-a_{21} .\] 
It is easy to see that for all monic quadratic polynomials $q(t)=t^2+at+b$, there exists $B$ such that 
\[ p_{A+B}(t)=q(t) .\] 
For example, we can take $a_{11}=-a, a_{21}=-b, a_{12}=a_{22}=0$.

We see in particular that $q(t)=t^2+1$ violates the above inequality at $\lambda=0$. However, there is still $B$ such that $p_{A+B}(t)=q(t)$. 
\end{expl} 

This example shows that it is possible to find more relaxed conditions so that Theorem \ref{thm:main} still holds true. We begin the derivation of the necessary and sufficient conditions on $q(t)$ with a lemma.

\begin{lem} \label{lem:power}
Let R be a ring and let $A,B\in R$. Then 
\[ (B+A)^k=\left[ \sum_{m=0}^{k-1} A^m B (B+A)^{k-m-1} \right]+A^k .\] 

\end{lem} 

\begin{proof}
Let us prove this by induction. For $k=1$, the left hand side and the right hand side are both $B+A$. Let's consider $k=2$. The left hand side is 
\[ (B+A)^2=(B+A)(B+A)=B(B+A)+A(B+A)=B(A+B)+AB +A^2.\]

Suppose the formula is true for $k$. Let us show that it is true for $k+1$. Indeed we have 
\begin{align*}
(B+A)^{k+1} &=(B+A)(B+A)^k=B(B+A)^k+A (A+B)^k \\ 
                      &=B(B+A)^k+ A \left[ A^k+ \sum_{m=0}^{k-1} A^m B (B+A)^{k-m-1}  \right]\\ 
                      &=B(B+A)^k+A^{k+1}+\sum_{m=0}^{k-1} A^{m+1}B(B+A)^{k-m-1}.
\end{align*} 
For the last term, let $n=m+1$ 
\[ \sum_{m=0}^{k-1} A^{m+1}B(B+A)^{k-m-1}=\sum_{n=1}^{k} A^n B (B+A)^{(k+1)-n-1}.\]
Therefore, we can see that 
\[ (B+A)^{k+1}=\left[ \sum_{m=0}^{k} A^m B (B+A)^{k-m} \right]+A^{k+1} .\] 
By induction, the above formula is true for all $k$. 
\end{proof}
We provide another pictorial proof for Lemma \ref{lem:power}.

\begin{proof}
Terms in the expression of $(A+B)^k$ corresponds to paths of length $k-1$ in the following labelled graphs (with $2k$ nodes).

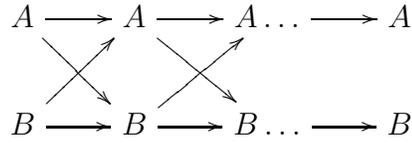
\xymatrix{
&&&& A \ar[r] \ar[rd] & A \ar[r] \ar[rd] & A \ldots  \ar[r] & A\\
&&&& B \ar[r] \ar[ru] & B \ar[r] \ar[ru] & B \ldots \ar[r] & B \\
}
Apart from the term $A^k$, each term has an initial block of the form $A^mB$, $0 \leq m \leq k-1$. Visualizing that block in the graph above (starting from the left), and considering all terms which begin with that block, we see that they correspond to continuing paths through the expansion $(B+A)^{k-m-1}$.
\end{proof}
Here is a direct corollary of Lemma \ref{lem:power}.
\begin{cor} \label{cor:rank}
For each $k$ 
\[ \rank((A+B)^k) \leq k \rank(B)+\rank(A^k).\]

\end{cor}

\begin{proof}
This is a direct consequence of lemma 1 and the facts that for two matrices $M,N$ 
\[ \rank(MN) \leq \min \{\rank(M), \rank(N) \}, \]
and 
\[ \rank(M+N) \leq \rank(M)+\rank(N) .\] 
\end{proof}

\begin{rem}
After writing this article, we learned that Lemma \ref{lem:power} and Corollary \ref{cor:rank}  have been discussed in \cite[Theorem 2.2]{[MMRR]}. We decide to keep those statements here in order to make this article self-contained. 
\end{rem} 

We recall the following notation which was mentioned in the introduction. Let $A$ be a matrix and $\lambda_0 \in F$. We denote by $\alg_{\lambda_0}(A)$ the algebraic multiplicity of $\lambda_0$ with respect to $A$. More precisely, 
\[ \alg_{\lambda_0}(A)=m_{\lambda_0}(p_{A}(t)) .\]

\begin{prop}
Let $B$ be a rank 1 matrix. Recall $j_{\lambda}(A)$ be the largest size of the Jordan block of $A$ of the form 
\[ J_{\lambda, r} = 
\begin{pmatrix}
\lambda& 1            & \;     & \;  \\
\;        & \lambda    & \ddots & \;  \\
\;        & \;           & \ddots & 1   \\
\;        & \;           & \;     & \lambda      
\end{pmatrix}. \]

Then 
\[ \rank((B+A-\lambda)^{j_{\lambda}})) \leq j_{\lambda}(A) \rank(B)+n-\alg_{\lambda}(A)=j_{\lambda}(A)+n-\alg_{\lambda}(A) .\] 
In particular, $\lambda$ is an eigenvalue of $B+A$ with algebraic multiplicity at least $\alg_{\lambda}(A)-j_{\lambda}$. 

\end{prop} 
\begin{proof}
It is enough to prove the above statement when $\lambda=0$.  In this case, all Jordan blocks of $A$ with $0$ on the diagonal will become zero when we raise $A$ the $j_{0}$-power. We then see that 
\[ \rank(A^{j_{0}})=n-\alg_{0}(A) .\]
Therefore, by Corollary $1$, we get the inequality 

\[ \rank((B+A-\lambda)^{j_{0}})) \leq j_{0}(A)+n-\alg_{0}(A) .\] 

The second statement is a direct consequence of this inequality. 
\end{proof}
A direct consequence of this proposition is the following.
\begin{cor} \label{cor:condition}
Suppose that $q(t)=p_{A+B}(t)$ for some rank $1$ matrix $B$. Then for all $\lambda \in F$ 
\[ m_{\lambda}(q) \geq \alg_{\lambda}(A)-j_{\lambda}(A).\]

\end{cor}

The matrix $q(t)$ in Example \ref{expl:first_expl} has this property. In the next section we will show that in fact this condition is also sufficient.

\section{Sufficient conditions} 
In this section, we show that the conditions given in Corollary \ref{cor:condition} are in fact sufficient. To show this, we first introduce some lemmas in matrix algebra. 

\begin{lem}(Weinstein–Aronszajn formula, see \cite[Proposition 11]{[Chervov]}) \label{lem:WA_formula}
Let $M$ be a $m \times n$ matrix and $N$ is an $n \times m$ matrix. Then 
\[ \det(I_m-MN)=\det(I_n-NM) .\] 

\end{lem} 

\begin{lem} \label{lem:factorization}
Let $B=vw^{T}$ where $v, w \in F^{n}$. Let $p_{A+B}(t)$ and $p_{A}(t)$ be the characteristic polynomial of $A+B$ and $A$ respectively. Then 
\[ p_{A+B}(t)=p_{A}(t) (1-w^t(tI-A)^{-1}v) .\] 
Equivalently 

\[ \frac{p_{A+B}(t)}{p_{A}(t)}=1-w^{t} (tI-A)^{-1} v .\] 

\end{lem} 

\begin{proof}
We have 
\begin{align*}
p_{A+B}(t) &=\det(tI-B-A) \\
     &=\det(tI-A) \det(I-(tI-A)^{-1}B)\\
 	&=\det(tI-A) \det(I - (tI-A)^{-1}vw^t )\\
 	&=p_{A}(t) (1-w^t(tI-A)^{-1}v)\\
\end{align*}

Note that the third equality follows from Lemma \ref{lem:WA_formula} by taking $M=(tI-A)^{-1}v$ and $N=w^t$.
\end{proof}

By this lemma, we see that the condition $p_{A+B}(t)=q(t)$ is equivalent to 
\[ 1-w^{t}(tI-A)^{-1} v=\frac{q(t)}{p_{A}(t)}.\] 
In other words  
\[ w^{t} (tI-A)^{-1} v= \frac{p_{A}(t)-q(t)}{p_A(t)}=\frac{h(t)}{p_A(t)} ,\]
with $h(t)=p_{A}(t)-q(t)$. Note that because $p_{A}(t), q(t)$ are both monic polynomials of degree $n$, $h(t)$ is a polynomial of degree at most $n-1$.

By choosing a suitable basis of generalized eigenvectors for $A$, we can assume that in the standard basis of $F^n$, $A$ is a direct sum of Jordan blocks $J_{\lambda,r}$, namely 
\[ A=\bigoplus_{\lambda, r} J_{\lambda,r } .\]
Here $J_{\lambda, r}$ is a basic Jordan block of the form
\[ J_{\lambda, r} = 
\begin{pmatrix}
\lambda& 1            & \;     & \;  \\
\;        & \lambda    & \ddots & \;  \\
\;        & \;           & \ddots & 1   \\
\;        & \;           & \;     & \lambda      
\end{pmatrix}. \]

We will consider different cases.
\subsection{$A$ is a single Jordan block}
We first consider the easiest case, namely $A$ is a single Jordan block $A=J_{\lambda, n}$.  In this case, the corresponding condition in Corollary \ref{cor:condition} is simply 
\[ m_{\lambda}(q) \geq n-n=0 .\] 
We see that all monic polynomials of degree $n$ satisfy this condition. We will show that in fact, there always exists a rank $1$ matrix $B$ such that 
\[ p_{A+B}(t)=q(t).\]

Let us write 
\[ A= \lambda I_n +N ,\]
where $N$ is the nilpotent matrix 
\[ N = 
\begin{pmatrix}
0& 1            & \;     & \;  \\
\;        & 0    & \ddots & \;  \\
\;        & \;           & \ddots & 1   \\
\;        & \;           & \;     & 0      
\end{pmatrix}. \]

We need the following lemma. 
\begin{lem} \label{lem: inverse}
Let $A=J_{\lambda, n}=\lambda I_n +N $ with $N$ is the above nilpotent matrix. Then 
\[ (tI_n-A)^{-1}= \sum_{i=0}^{n-1} \frac{N^i}{(t-\lambda)^{i+1}} .\]
\end{lem} 
\begin{proof}
We recall that if $x,y$ are two commuting matrices of the same size then 
\[ x^n-y^n=(x-y) \left[\sum_{i=0}^{n-1} x^{n-i-1}y^{i} \right] .\]
Apply this formula for $x=(t-\lambda) I_n, y=N$ and note that $N^n=0$, we have  
\begin{align*}
(t-\lambda)^n I_n &=(t-\lambda)I_n^n - N^n  \\ 
   &=((t-\lambda)I_n-N) \left[ \sum_{i=0}^{n-1} (t-\lambda)^{n-1-i} I_n^{n-1-i} N^{i} \right] \\
   &=((t-\lambda)I_n-N) \left[ \sum_{i=0}^{n-1} (t-\lambda)^{n-i-i} N^{i} \right].
\end{align*}
Consequently 
\begin{equation} \label{eq:inverse}
((t-\lambda)I_n-N) \left[ \sum_{i=0}^{n-1} \frac{N^{i}}{(t-\lambda)^{i+1}} \right]=I_n.
\end{equation}
Note that $(t-\lambda)I_n-N=(tI_n-A)$. Therefore, the above lemma follows directly from Equation \ref{eq:inverse}.
\end{proof}

We are now ready to prove the following.
\begin{prop} \label{prop:single_block}
Suppose $A=J_{\lambda, n}$ and $q(t)$ is a monic polynomial of degree $n$. Then there exists a rank $1$ matrix $B=vw^{t}$ such that 
\[ p_{A+B}(t)=q(t) .\]

More concretely, let $\{e_1, e_2, \ldots, e_n \}$ be the standard basis for $F^n$. Then we can take $v=e_n$ and 
\[ w=\sum_{i=1}^n a_{n-i} e_i ,\] 
where $a_i$ is determined explicitly by $q(t)$.

\end{prop} 
\begin{proof}
Since $A=J_{\lambda, n}$ we have 
\[ p_{A}(t)=(t-\lambda)^n .\]

By Lemma \ref{lem:factorization}, it is enough to show that there exists $v,w \in F^n$ such that 
\begin{equation} \label{eq:equality}
w^t (tI-A)^{-1} v= \frac{h(t)}{p_{A}(t)}=\frac{h(t)}{(t-\lambda)^n} ,
\end{equation}
with $h(t)=p_{A}(t)-q(t)$. Since $h(t)$ is a polynomial of degree at most $n-1$, by taking the Taylor expansion of $h(t)$ at $\lambda$ we have 
\[ h(t)=\sum_{i=0}^{n-1} a_i (t-\lambda)^{n-1-i} .\]
Therefore Equation \ref{eq:equality} is equivalent to
\begin{equation} \label{eq:second_equality}
w^{t}(tI-A)^{-1} v= \sum_{i=0}^{n-1} \frac{a_i}{(t-\lambda)^{i+1} }. 
\end{equation} 
Let us take 
\[ w=\sum_{i=0}^{n-1} a_i e_{n-i},  \quad v=e_n .\]
We claim that $v, w$ satisfy the condition given in Equation \ref{eq:second_equality}. In fact, by Lemma \ref{lem: inverse} we have 
\[ (tI-A)^{-1} =\sum_{i=0}^{n-1} \frac{N^i}{(t-\lambda)^{i+1}} .\]
Therefore 
\begin{align*} 
w^{t} (tI-A)^{-1}v &= w^t \left[\sum_{i=0}^{n-1} \frac{N^i e_n}{(t-\lambda)^{i+1}} \right] \\ 
 &= \sum_{i=0}^{n-1} \frac{w^t e_{n-i}}{(t-\lambda)^{i+1}}\\
 &=\sum_{i=0}^{n-1} \frac{a_i}{(t-\lambda)^{i+1}}.
\end{align*}
Here we use the fact that $N^i e_n=e_{n-i-1}$ and $w^t e_{n-i}=a_i$.

\end{proof}
\subsection{A is the direct sum of Jordan blocks with different $\lambda_i$}
Next, we consider the case where $A$ is a direct sum of these blocks with different $\lambda$, namely 
\[ A=\bigoplus_{i=1}^m J_{\lambda_i, n_i}, \]
with $\lambda_i$ are pairwisely different. We note that in this case, the condition in Corollary \ref{cor:condition} is simply 
\[ m_{\lambda_i}(q) \geq n_i-n_i=0 .\]
In other words, all $q(t)$ have this property. In fact, we have the following proposition 
\begin{prop} \label{prop:distinct_block}
Suppose \[ A=\bigoplus_{i=1}^m J_{\lambda_i, n_i}, \]
with $\lambda_i$ are pairwisely different. Let  $q(t)$ is a monic polynomial of degree $n$. Then there exists a rank $1$ matrix $B=vw^{t}$ such that 
\[ p_{A+B}(t)=q(t) .\]
\end{prop} 

\begin{proof}
In this case, we have  
\[ p_{A}(t)=\prod_{i=1}^m (t-\lambda_i)^{n_i} .\] 
As before, we need to find $v, w$ such that

\[ w^t (tI-A)^{-1} v= \frac{h(t)}{p_{A}(t)}=\frac{h(t)}{\prod_{i=1}^m(t-\lambda_i)^{n_i}}  .\]
with $h(t)=p_{A}(t)-q(t)$. We need the following lemma. 
\begin{lem}
There exist polynomials $h_i(t)$ such that $\deg(h_i(t)) \leq n_i-1$ and 
\[ \frac{h(t)}{q(t)}=\sum_{i=1}^m \frac{h_i(t)}{(t-\lambda_i)^{n_i}} .\]
\end{lem} 
\begin{proof}
By the Chinese remainder theorem, there exist polynomials $g_i(t) \in F[t]$ such that \[ \frac{h(t)}{q(t)}=\sum_{i=1}^m \frac{g_i(t)}{(t-\lambda_i)^{n_i}} .\]
By the Euclidean algorithm, we can write 
\[ \frac{g_i(t)}{(t-\lambda_i)^{n_i}}=\frac{h_i(t)}{(t-\lambda_i)^{n_i}}+r_i(t) ,\]
where $h_i(t), r_i(t)$ are a polynomials and the degree of $h_i(t)$ is at most $n_i-1$. Hence we have 
\[ \frac{h(t)}{q(t)}=r(t)+\sum_{i=1}^m \frac{h_i(t)}{(t-\lambda_i)^{n_i}} ,\]
with $r(t)=r(t)+\ldots+r_{m}(t)$. Since $\deg(h)<\deg(q)$, we must have $r(t)=0.$ Consequently,  
\[ \frac{h(t)}{q(t)}=\sum_{i=1}^m \frac{h_i(t)}{(t-\lambda_i)^{n_i}} .\]

\end{proof}

Let us denote by $A_{\lambda}$ the generalized eigenspace associated with $\lambda$, namely 
\[ A_{\lambda}=\{ v \in F^n| (A-\lambda)^m v=0, \text{for some $m$} \} .\]
We see that , if $z_i \in A_{\lambda_i}$, $z_j \in A_{\lambda_j}$, and $\lambda_{i} \neq \lambda_{j}$ then 
\[ z_i^t z_j=z_j^t z_i=0. \] 
Let us write 
\[ v=\sum_{i=1}^m v_i, \quad w=\sum_{i=1}^m w_i .\]

with $v_i, w_i \in A_{\lambda_i}$. Since $(tI-A)^{-1} v_i \in A_{\lambda_i}$ and $w_j \in A_{\lambda_j}$ we see that if $i \neq j$ then
\[ w_j^t (tI-A)^{-1} v_i=0 .\]
Therefore 
\begin{equation} \label{eq:composition}
w^t (tI-A)^{-1}v= \sum_{i=1}^m \overline{w}_{i}^t (tI_{n_i}-J_i)^{-1} \overline{v}_{i} .
\end{equation}
Here for a vector $z_i \in A_{\lambda_i}$, $\overline{z}_i$ is the projection of $z_i$ to the component corresponding to $J_{\lambda_i}$. More precisely, suppose 
\[ z_i=(0, \ldots, 0, \underbrace{m_1, m_2, \ldots, m_{n_i}}_{\text{$i$-th component}}, \ldots 0 ) \in F^n,\]
then 
\[ \overline{z}_i=(m_1, m_2, \ldots, m_{n_i}) \in F^{n_i} .\]
Note that, we can recover $z_i$ from $\overline{z}_i$ as well.

Now, by Proposition \ref{prop:single_block}, we know that we can find $\overline{v}_i, \overline{w}_i$ (and hence $v_i, w_i$) such that that 
\[ \overline{w}_{i}^t (tI_{n_i}-J_i)^{-1}=\frac{h_i(t)}{(t-\lambda_i)^{n_i}} .\] 
From Equation \ref{eq:composition}, we see that 
\[ w^t (tI-A)^{-1}v= \sum_{i=1}^m \overline{w}_{i}^t (tI_{n_i}-J_i)^{-1} \overline{v}_{i} . \]
This completes the proof in this special case. 
\end{proof} 
\subsection{The general case}

Finally, we consider the general case. 
\begin{prop}
Let $A$ be an $n \times n$ matrix. Suppose $q(t)$ be a monic polynomial of degree $n$ satisfying the condition in Corollary \ref{cor:condition}, namely 
\[ m_{\lambda}(q) \geq \alg_{\lambda}(A)-j_{\lambda}(A) .\] 

Then, there exists a rank $1$ matrix $B$ such that 
\[ p_{A+B}(t)=q(t) .\] 
\end{prop} 

\begin{proof}
As explained above, it is enough to consider the case where $A$ is a direct sum of Jordan blocks (possibly with repeated $\lambda_i$). Let $\{\lambda_1, \lambda_2, \ldots, \lambda_m \}$ be the set of distinct eigenvalues of $A$. We can then decompose $A$ as 
\[ A=A_1 \oplus A_2 ,\]
where 
\[ A_1= \bigoplus_{i=1}^m J_{\lambda_i, j_{\lambda_i}} ,\]
and $A_2$ is the rest. In other words, for each eigenvalue $\lambda$, $A_1$ contains exactly one copy of $J_{\lambda,r}$ with $r$ being the largest size of all Jordan blocks in $A$ of the form $J_{\lambda, r_k}$.  We note, by the assumption, we can write 
\[ q(t)=q_1(t) \prod_{i=1}^m (t-\lambda_i)^{\alg_{\lambda_i}(A)-j_{\lambda_i}(A)} ,\]
where $q_1(t)$ is a monic polynomial of degree $d=\sum_{j=1}^n j_{\lambda_j}(A)$. By Proposition \ref{prop:distinct_block}, we know that there exists a rank $1$ matrix $B_1$ of size $d \times d$ such that 
\[ p_{A_1+B_1}(t)=q_1(t) .\]

Let $B=B_1 \oplus 0_{n-d}$ be a matrix of size $n \times n$. Note that, $B$ has rank $1$ as well. Furthermore 
\[ p_{A+B}(t)=p_{A_1+B_1}(t) p_{A_2+B_2}(t)=q_1(t) p_{A_2+B_2}(t) .\]
It is easy to see that 
\[ p_{A_2+B_2}(t) =\prod_{i=1}^m (t-\lambda_i)^{\alg_{\lambda_i}(A)-j_{\lambda_i}(A)} .\] 
We then conclude that 
\[ p_{A+B}(t)=q(t) .\] 
\end{proof}


\begin{thebibliography}{99}
\bibitem{[Bau]}
{H. Baumgartel, Analytic Perturbation Theory for Matrices and Operators, Oper. Theory Adv. Appl. 15, Birkhäuser, Basel (1985).}

\bibitem{[BM]}
{J. Bénasséni and A. Mom, 
Inequalities for the eigenvectors associated to extremal eigenvalues in rank one perturbations of symmetric matrices, Linear Algebra Appl. 570 (2019), 123–137.}

\bibitem{[CNg]} 
{W. S. Cheung and T. W. Ng, Relationship between the zeros of two polynomials, Linear Algebra Appl. 432 (2010), no. 1, 107–115. MR 2566462}
\bibitem{[Chervov]}
{A. Chervov, G. Falqui and V. Rubtsov, Algebraic Properties of Manin Matrices 1, Adv. Appl. Math. 43 (2009) 239.}

\bibitem{[DZ]}
{J. Ding and A. Zhou, Eigenvalues of rank-one updated matrices with some applications, Appl. Math. Lett. 20 (2007), no. 12, 1223–1226.}

\bibitem{[Kato]}
{T. Kato, Perturbation Theory for Linear Operators, Springer, NewYork, NY(1966).}

\bibitem{[Kru]}
{M. Krupnik, Changing the spectrum of an operator by perturbation.
Sixth Haifa Conference on Matrix Theory (Haifa, 1990), Linear Algebra Appl. 167 (1992), 113–118.}

\bibitem{[KT]}
{ O. Kushel and M. Tyaglov, Circulants and critical points of polynomials. J. Math. Anal. Appl. 439 (2016), no. 2, 634–650.}
\bibitem{[Lidskii]}
{V.B. Lidskii, To perturbation theory of non-self adjoint operators, USSR Comput. Math. Math. Phys. 6, 52–60 (1966).}

\bibitem{[MMRR]}
{C. Mehl, V. Mehrmann, A.C.M. Ran, L. Rodman
Eigenvalue perturbation theory of classes of structured matrices under generic structured rank one perturbations, Linear Algebra Appl., 435 (2011), pp. 687-716.}


\bibitem{[RW]}
{A. Ran and M. Wojtylak, Eigenvalues of rank one perturbations of unstructured matrices, Linear Algebra Appl. 437 (2012), no. 2, 589–600.}

\bibitem{[Wi]}
{Wilkinson, J. H, The algebraic eigenvalue problem Clarendon Press, Oxford 1965.}

\bibitem{[Zhang]}
{F. Zhang, Matrix Theory: Basic Results and Techniques (2nd ed.), New York, Springer, 2010. }

\end{thebibliography}
\end{document}